\documentclass[letterpaper, 10 pt, conference]{ieeeconf}  

\IEEEoverridecommandlockouts                              
\usepackage{additionalSetting}
\usepackage{anyfontsize}



\makeatletter
\newcommand\fs@spaceruled{\def\@fs@cfont{\bfseries}\let\@fs@capt\floatc@ruled
	\def\@fs@pre{\vspace{1.5mm}\hrule height.8pt depth0pt \kern2pt}%
	\def\@fs@post{\kern2pt\hrule\relax}%
	\def\@fs@mid{\kern2pt\hrule\kern2pt}%
	\let\@fs@iftopcapt\iftrue}
\makeatother

\floatstyle{spaceruled}
\restylefloat{algorithm}

\newcommand\dvr[1]{{\color{blue}#1}}

\title{\LARGE \bf
Robust Differential Dynamic Programming
}

\author{Dennis Gramlich, Carsten W. Scherer, Christian Ebenbauer
\thanks{This work was supported by Sener Aerospacial under a program of, and funded by, the European Space Agency.}
\thanks{Dennis Gramlich and Christian Ebenbauer are with the Chair of Intelligent Control Systems,
        RWTH-Aachen,
        D-52074 Aachen, Germany
        {\tt\small \{dennis.gramlich,christian.ebenbauer\} @ic.rwth-aachen.de}}%
\thanks{Carsten W. Scherer is with the Chair of Mathematical Systems Theory,
    University of Stuttgart, 70174 Stuttgart, Germany
        {\tt\small carsten.scherer@mathematik.uni-stuttgart.de}}%
\thanks{Funded by Deutsche Forschungsgemeinschaft (DFG, German Research Foundation) under Germany's Excellence Strategy - EXC 2075 – 390740016. We acknowledge the support by the Stuttgart Center for Simulation Science (SimTech).}
}%

\begin{document}

\maketitle
\thispagestyle{empty}
\pagestyle{empty}

\begin{abstract}

Differential Dynamic Programming is an optimal control technique often used for trajectory generation. 
Many variations of this algorithm have been developed in the literature, including algorithms for  stochastic dynamics or state and input constraints.  In this contribution, we develop a robust version of Differential Dynamic Programming that uses generalized plants and multiplier relaxations for uncertainties. To this end, we study a version of the Bellman principle and use convex relaxations to account for uncertainties in the dynamic program. The resulting algorithm can be seen as a robust trajectory generation tool for nonlinear systems.

\end{abstract}

\section{INTRODUCTION}
\label{sec:1}

In this paper, we deal with the robust planning of trajectories for linear and nonlinear dynamical systems. Trajectory planning or feed-forward control is a research field probably at least as old as control engineering itself. It is thus surprising how rarely robustness of trajectory generation has been considered in the literature. We address this problem by combining Differential Dynamic Programming (DDP) methods with robust control. 

DDP has been developed in the 1960s \cite{OP_Mayne1966} as an approximate Dynamic Programming tool that bridges the gap between Dynamic Programming and local optimization methods for trajectory planning. 
Since the development of DDP, different extensions of these algorithms have emerged. This includes stochastic DDP \cite{theodorou2010stochastic,pan2014probabilistic,ozaki2018stochastic} for stochastic optimal control problems, DDP methods for constrained optimization problems \cite{OP_Pavlov2020} and newer extensions such as \emph{sampled Differential Dynamic Programming} \cite{rajamaki2016sampled}. In this article, we treat uncertainties in the Dynamic Program in a worst-case fashion, i.e., we maximize over the uncertainties to provide a robust version of DDP. This results in a minimax optimization problem in the Dynamic Program just like in \emph{minimax Differential Dynamic Programming} \cite{sun2018min}. Unlike minimax DDP, we make use of convex relaxation techniques from robust control such that our approach applies to (nonlinear) generalized plants with parametric or dynamic uncertainties and not just disturbances as in existing minimax DDP. This approach allows us to carry over the multiplier relaxation techniques from robust control to DDP. Indeed, an alternative perspective on our work is the development of robust control theory for time varying finite horizon control problems, as it is suggested in \cite{buch2021finite} for linear dynamical systems. Such finite horizon perspectives enable the design of robust time varying feedback controllers for nonlinear systems linearized around a reference trajectory, which is already of high value by itself in our opinion. Compared to \cite{buch2021finite}, we emphasize that our algorithm simultaneously optimizes a feedback term and the reference trajectory in a robust control fashion.

Most of the works dealing with robustness in trajectory generation have been conducted in Model Predictive Control (MPC). Despite being mostly concerned with fixed point stabilization, MPC can be viewed as control by sequential trajectory generation.
Inherent robustness properties of MPC have been studied in \cite{de1996robustness,scokaert1997discrete}, but for the general nonlinear case, non-robustness of nominal MPC has been reported e.g. in \cite{grimm2004examples}. Resulting from these investigations, robust tube-based MPC has been developed to address bounded disturbances \cite{mayne2005robust}. Further, robust control of nonlinear systems is studied in the tube-based MPC setting \cite{mayne2011tube}. Still, the fusion of MPC with classical multiplier relaxations from robust control remains an open issue. First steps into this direction (see \cite{schwenkel2020dynamic}) are still limited to fixed multipliers and linear systems. An alternative to tube-based MPC to be mentioned is minimax MPC \cite{campo1987robust}, where similar to our ideas, convex relaxations in the form of linear matrix inequalities are used to handle robust minimax optimization problems.

A major issue with approaches such as minimax MPC is computational complexity. Since convex optimization allows for more efficient solutions, it is therefore typically desirable to \emph{convexify} trajectory optimization and robust controller synthesis problems (compare \cite{OP_Acikmecse2013,OP_Leeman2021}). For this purpose, we leverage the S-procedure and dualization techniques from robust control to formulate our DDP iterations as LMI optimization problems.

Summarizing, we makes the following contributions.
 \begin{itemize}
     \item We combine DDP with robust control to address trajectory planning problems for systems in generalized plant form with parametric or dynamic uncertainties.
     \item We relax the maximization over the uncertainties, such that the solution approach is completely convex for linear systems and a sequential convex programming approach for nonlinear systems.
     \item We optimize over affine linear feedback policies, i.e., we simultaneously optimize (robustly) over a trajectory and a linear feedback around this trajectory.
 \end{itemize}
\section{Problem statement}
\label{sec:2}
Consider the following minimax optimization problem
\begin{align}
    \minimize_{(\pi_t(\cdot))} \maximize_{(\Delta_t(\cdot))\subseteq \ovl{\Delta}} ~&~ \sum_{t=0}^{T-1} f_0(x_t,u_t) + V_T(x_T) \label{eq:origOpt}\\
    \mathrm{subject~to} ~&~ \begin{pmatrix}
                x_{t+1}\\
                z_t\\
            \end{pmatrix}
            =
            \begin{pmatrix}
                    f(x_t,u_t, w_t)\\
                    g(x_t,u_t,w_t)
            \end{pmatrix}\nonumber\\
            &~ w_t = \Delta_t(z_t)\nonumber\\
            &~ u_t = \pi_t(x_t),\nonumber\\
            &~ x_0 = \bar{x} \nonumber
\end{align}
for a fixed initial condition $\bar{x}$, where the functions $f(\cdot), g(\cdot), f_0(\cdot)$ and $V_T(\cdot)$ are continuously differentiable and all constraints have to apply for $t = 0,\ldots, T-1$.
In the present paper, we interpret
\begin{align}
    \begin{pmatrix}
                    x_{t+1}\\
                    z_t\\
            \end{pmatrix}
            &=
            \begin{pmatrix}
                    f(x_t,u_t, w_t)\\
                    g(x_t,u_t,w_t)
            \end{pmatrix},
            &w_t = \Delta_t(z_t) \label{eq:interconnection}
\end{align}
as a nonlinear generalized plant with state $x_t \in \bbR^n$, control input $u_t \in \bbR^m$, disturbance input $w_t \in \bbR^d$ and output $z_t \in \bbR^l$. Alternatively, $w_t$ can be seen as the input of the adversarial uncertainty policy $\Delta_t$. Accordingly, $(\pi_t(\cdot))$ denotes the control policy consisting of the $T$ mappings $\pi_t: \bbR^n \to \bbR^m$ for $t = 0,\ldots , T-1$ and $(\Delta_t(\cdot))$ is an uncertain component consisting of $\Delta_t : \bbR^{l} \to \bbR^{d}$ for $t = 0,\ldots, T-1$, contained in a set $\ovl{\Delta}$ defined by
\begin{align*}
    \ovl{\Delta} = \{ \Delta : \bbR^{l} \to \bbR^{d} \mid \Delta(z) \in \calC(z) ~ \forall z \in \bbR^l \},
\end{align*}
where $0 \in \calC(z) \subsEq \bbR^{d}$ for all $z \in \bbR^{l}$. We assume that the interconnection \eqref{eq:interconnection} is well-posed, i.e., for any $x\in \bbR^n$, $u \in \bbR^m$ and $\Delta \in \ovl{\Delta}$, there exists exactly one $(w,z) \in \bbR^{d}\times\bbR^{l}$, such that $z = g(x,u,w), w = \Delta (z)$ is satisfied.

The aim of this paper is to provide a methodology for computing approximate solutions to the optimization problem \eqref{eq:origOpt} in terms of robust feedback policies $(\pi_t(\cdot))_{t=0}^{T-1}$. In the nominal case ($w_t = 0$ for all $t$), this problem is addressed by DDP. To address the robust optimization problem \eqref{eq:origOpt}, we extend DDP with techniques from robust control, i.e., we make use of convex relaxations to relax the maximization in problem \eqref{eq:origOpt}. In this course, we derive a new methodology for computing policies to optimize the robust performance in \eqref{eq:origOpt}.

\section{Robust Differential Dynamic Programming}
\label{sec:3}

In the following, we show how approximate solutions to problem \eqref{eq:origOpt} can be obtained. To this end, we follow the lines of DDP to derive a sequential convex programming procedure. To facilitate the distinction between decision variables (and those quantities containing decision variables) and constants in the subsequent derivations and convexifications, we highlight decision variables in blue in the following two sections. This excludes the variables $(x,u,w), (x_t,u_t,w_t), (\delta_t^x,\delta_t^u,\delta_t^w)$, since these are eliminated in the derivation. 
A solution to \eqref{eq:origOpt} can be obtained from a solution $(Q_t(\cdot,\cdot),V_t(\cdot),\pi_t(\cdot))$ of the Bellman equation/Dynamic Programming Recursion
\begin{align*}
    Q_t(x,u,w) &= f_0(x,u) + V_{t+1}(f(x,u,w)),\\
    \pi_t(x) &\in \argmin_{u \in \bbR^m} \max_{\substack{\Delta_t \in \ovl{\Delta}\\
    w = \Delta_t(g(x,u,w))}} Q_t(x,u,w),\\
    V_t(x) &= \max_{\substack{\Delta_t \in \ovl{\Delta}\\
    w = \Delta_t(g(x, \pi_t(x),w))}} Q_t(x,\pi_t(x),w),
\end{align*}
for $t = 0,\ldots,T-1$. DDP approximates this process by alternately computing quadratic approximate solutions of the Bellman equation around a trajectory (backward pass) and computing new trajectories (forward pass).

\textbf{Backward pass:} For $t = T-1,\ldots,0$ and $\widetilde{V}_{T}(x) = V_T(x)$ the following steps are executed:
\begin{compactitem}
    \item Computation of a second-order approximation of the $Q$-function around the trajectory $(x_t,u_t,w_t)_{t=0}^T$:
    \begin{align}
        f_0(x, u) + \widetilde{V}_{t+1}(f(x, u, w)) \approx
        \widetilde{Q}_t(x, u, w) \nonumber\\
        =:
        \begin{pmatrix}
            1\\ \delta_t^x\\ \delta_t^u\\ \delta_t^w
        \end{pmatrix}^\top
        \begin{pmatrix}
            q_t^{11} & q_t^{12} & q_t^{13} & q_t^{14}\\
            q_t^{21} & \bQ_t^{22} & \bQ_t^{23} & \bQ_t^{24}\\
            q_t^{31} & \bQ_t^{32} & \bQ_t^{33} & \bQ_t^{34}\\
            q_t^{41} & \bQ_t^{42} & \bQ_t^{43} & \bQ_t^{44}
        \end{pmatrix}
        \begin{pmatrix}
            1\\ \delta_t^x\\ \delta_t^u\\ \delta_t^w
        \end{pmatrix}, \label{eq:quadraticQ}
    \end{align}
    where $\delta_t^x = x - x_t$, $\delta_t^u = u - u_t$, $\delta_t^w = w - w_t$.
    \item Computation of the policy $(\tilde{\pi}_t)_{t=0}^{T-1}$ and value function approximation $(\widetilde{V}_t)_{t=0}^{T-1}$:
    \begin{align}
        \tilde{\pi}_t(x) &:= u_t + \dvr{k_t^1} + \dvr{\bK_t^2}\delta_t^x \label{eq:optimizePolicy}\\
        &\approx \argmin_{u \in \bbR^n} \max_{\substack{\Delta_t \in \ovl{\Delta}\\
    w = \Delta_t(g(x,u,w))}} \widetilde{Q}_t(x,u,w), \nonumber\\
        \widetilde{V}_t(x) &
        := \max_{\substack{\Delta_t \in \ovl{\Delta}\\
        		w = \Delta_t(g(x,\pi_t(x),w))}}
        	\widetilde{Q}_t(x,\pi_t(x),w)\nonumber\\
        &\approx \begin{pmatrix}
                1\\
                \delta_t^x
            \end{pmatrix}^\top
            \begin{pmatrix}
                \dvr{p_t^{11}} & \dvr{p_t^{12}}\\
                \dvr{p_t^{21}} & \dvr{\bP_t^{22}}
            \end{pmatrix} \begin{pmatrix}
                1\\
                \delta_t^x
            \end{pmatrix}. \label{eq:valueFcnApprox}
    \end{align}
    Here, we approximate the ideal policy $\pi_t(\cdot)$ by the affine linear policy $\tilde{\pi}_t(\cdot)$.
\end{compactitem}
\textbf{Forward pass:} For the given initial value $x_0 = \bar{x}$ and $t = 0,\ldots,T-1$, the following steps are executed:
\begin{compactitem}
    \item A new trajectory $(x_t^+,u_t^+,w_t^+)_{t=0}^{T-1}$ is evaluated based on the new policy $(\tilde{\pi}_t(\cdot))_{t =0}^{T-1}$:
    \begin{align*}
        x_{t+1}^{+} = f(x_t^+,u_t^+,w_t^+), \quad u_t^+ = \tilde{\pi}_t(x_t^+), \quad w_t^+ = 0.
    \end{align*}
\end{compactitem}
Notice that we set $w_t^+ = 0$ in the forward pass, while we consider worst-case disturbances in the backward pass. This is because the backward pass has a robustifying effect (Theorem \ref{thm:robustnessStep}), whereas the forward pass determines the trajectory around which we approximate the $Q$-function. Here, it makes sense to choose the value $w_t$, yielding the best approximation $\widetilde{Q}_t(\cdot)$, which we assume to be the nominal value $w_t = 0$. In general, $w_t \in \calC(z_t)$ could be chosen in a way that minimizes the maximum distance to other points in $\calC(z_t)$.

Stated as above, DDP is a conceptual scheme rather than a ready-to-implement algorithm, since the details of how the $Q$-function is approximated by $\widetilde{Q}$ and how the min-max policy $\tilde{\pi}_t(\cdot)$ is calculated remain unspecified. Possibilities to compute $\widetilde{Q}$, such as second order Taylor approximation, are discussed in Section \ref{sec:5}. Now, we show how the maximization over the constraint set $\Delta_t \in \ovl{\Delta}$ is carried out. To this end, we make use of the celebrated S-procedure 
to recast \eqref{eq:optimizePolicy} as an LMI optimization problem.

To employ the S-procedure, we assume that the signals $\delta_t^x,\delta_t^w$ and $\delta_t^u$ related through
        \begin{align}
            w_t + \delta_t^w &= \Delta_t(g(x_t + \delta_t^x,u_t+\delta_t^u,w_t + \delta_t^w)) & \Delta_t(\cdot) \in \ovl{\Delta} \label{eq:disturbanceConst}
        \end{align}
        satisfy the quadratic inequality
        \begin{align}
            \begin{pmatrix}
            1\\ \delta_t^x\\ \delta_t^u\\ \delta_t^w
        \end{pmatrix}^\top
        \begin{pmatrix}
            \dvr{m_t^{11}} & \dvr{m_t^{12}} & \dvr{m_t^{13}} & \dvr{m_t^{14}}\\
            \dvr{m_t^{21}} & \dvr{\bM_t^{22}} & \dvr{\bM_t^{23}} & \dvr{\bM_t^{24}}\\
            \dvr{m_t^{31}} & \dvr{\bM_t^{32}} & \dvr{\bM_t^{33}} & \dvr{\bM_t^{34}}\\
            \dvr{m_t^{41}} & \dvr{\bM_t^{42}} & \dvr{\bM_t^{43}} & \dvr{\bM_t^{44}}
        \end{pmatrix}
        \begin{pmatrix}
            1\\ \delta_t^x\\ \delta_t^u\\ \delta_t^w
        \end{pmatrix}
        \geq 0 \label{eq:conicConst}
        \end{align}
        for a family of matrices $\dvr{\bM_t} \in \calM$. Stated otherwise, we assume \eqref{eq:conicConst} is a relaxation of \eqref{eq:disturbanceConst}. For the remainder of this work, we assume that $\calM$ is of the generic conic form
        \begin{align}
            \calM &= \left\{ \sum_{i=1}^s \dvr{\lambda_i} \bM^{(i)} \mid \dvr{\lambda_i} \in \bbR_{\geq 0} \right\}
            \label{eq:genericMultipliers}
        \end{align}
    	with given symmetric $\bM^{(i)}$, $i = 1,\ldots , s$ and $s \in \bbN$.
        Many of the multipliers applied in robust control admit the form \eqref{eq:genericMultipliers}. We call the matrices $\bM^{(1)}, \ldots , \bM^{(s)}$ the generators of $\calM$. Clearly, with this definition, \eqref{eq:conicConst} can equivalently be demanded only on the set of generators of $\calM$, such that \eqref{eq:conicConst} boils down to a finite number of quadratic constraints.
        Therefore, we can apply standard semi-definite relaxations for non-convex quadratic programming to handle the maximization problem. Hence, we relax the maximization problem
        \begin{align*}
            \maximize_{\delta_t^w} \widetilde{Q}_t(x,u,w) ~~~ \mathrm{s.t.}~ \eqref{eq:conicConst}
        \end{align*}
        with its primal dual formulation
        {\footnotesize\begin{align}
            \minimize_{\lambda_i \geq 0}\maximize_{\delta_t^w} \widetilde{Q}_t(x,u,w) +
            \sum_{i = 1}^s
        \begin{pmatrix}
            1\\ \delta_t^x\\ \delta_t^u\\ \delta_t^w
        \end{pmatrix}^\top
        \dvr{\lambda_i}
        \bM^{(i)}
        \begin{pmatrix}
            1\\ \delta_t^x\\ \delta_t^u\\ \delta_t^w
        \end{pmatrix}.
    \label{eq:worstCaseQ}
        \end{align}}
        Since the maximization problem is a non-convex quadratic program, there can be a duality gap in this step. Notice that we obtain then an upper bound on the worst-case value of $\widetilde{Q}_t$, i.e., we overestimate the worst-case. 
        This overestimation is now used to compute the approximate value function $\widetilde{V}_t(\cdot)$ by demanding 
        \begin{align}
            \widetilde{V}_t(x) \geq \minimize_{\dvr{\bM_t} \in \calM,\delta_t^u} \maximize_{\delta_t^w} \ovl{Q}_t(x,u,w),
            \label{eq:constrainedBellman}
        \end{align}
        where
        \begin{align*}
        \ovl{Q}_t(x,u,w) = 
            \begin{pmatrix}
            1\\ \delta_t^x\\ \delta_t^u\\ \delta_t^w
        \end{pmatrix}^\top
        (\bQ_t +\dvr{\bM_t})
        \begin{pmatrix}
            1\\ \delta_t^x\\ \delta_t^u\\ \delta_t^w
        \end{pmatrix}.
        \end{align*} 
        In \eqref{eq:constrainedBellman}, we require $\widetilde{V}_t(\cdot)$ to be larger or equal to \eqref{eq:worstCaseQ} to make sure that $\widetilde{V}_t(\cdot)$ overestimates the worst-case costs. This overestimation is easily confirmed since the term involving $\dvr{\bM_t}$ is positive for all admissible $\delta_t^x,\delta_t^u,\delta_t^w$ due to \eqref{eq:conicConst} and removing the constrained \eqref{eq:disturbanceConst} can only increase the costs.
        The gain of this reformulation is that we eliminate the constraints imposed by \eqref{eq:disturbanceConst}.
        Substituting a quadratic function parametrized by $\dvr{\bP_t}$ for $\widetilde{V}_t(\cdot)$ yields
        \begin{align}
            &0 \geq \minimize_{\dvr{\bM_t} \in \calM,\delta_t^u} \maximize_{\delta_t^w}
            \begin{pmatrix}
            1\\ \delta_t^x\\ \delta_t^u\\ \delta_t^w
        \end{pmatrix}^\top
        \underbrace{(\bQ_t + \dvr{\bM}_t)}_{:= \dvr{\ovl{\bQ}_t}}
        \begin{pmatrix}
            1\\ \delta_t^x\\ \delta_t^u\\ \delta_t^w
        \end{pmatrix} \nonumber\\
        & \hspace{29mm}
        -
        \begin{pmatrix}
                1\\
                \delta_t^x
            \end{pmatrix}^\top
            \begin{pmatrix}
                \dvr{p_t^{11}} & \dvr{p_t^{12}}\\
                \dvr{p_t^{21}} & \dvr{\bP_t^{22}}
            \end{pmatrix} \begin{pmatrix}
                1\\
                \delta_t^x
            \end{pmatrix}. \label{eq:nonConvexMinMax}
        \end{align}
The sequence of steps, which we have just performed, is a variant of the S-procedure in robust control or Lagrange relaxation in Optimization. As we have stressed, Lagrange relaxation can be a conservative estimate and so can be the S-procedure.
        
        In \eqref{eq:nonConvexMinMax}, we require the inequality to be satisfied for all $\delta_t^x$, while taking the minimum over $\delta_t^u$ and the maximum over $\delta_t^w$. Fortunately, \eqref{eq:nonConvexMinMax} is quadratic in $\delta_t^u$ and $\delta_t^w$, such that we can solve the min-maximization analytically.
        
        If $\bQ_t^{44} + \dvr{\bM_t^{44}}$ is not negative semi-definite, then the solution for $\delta_t^w$ is unbounded. 
        If $\bQ_t^{44} + \dvr{\bM_t^{44}}\prec 0$, we get
        \begin{align}
            \delta_t^w = -(\dvr{\ovl{\bQ}_t^{44}})^{-1}\begin{pmatrix}
                \dvr{\bar{q}_t^{41}} & \dvr{\ovl{\bQ}_t^{42}} & \dvr{\ovl{\bQ}_t^{43}}
            \end{pmatrix}
            \begin{pmatrix}
                1\\
                \delta_t^x\\
                \delta_t^u
            \end{pmatrix}
            =: \widetilde{\Delta}_t(\delta_t^x,\delta_t^u).
            \label{eq:approxUncertainty}
        \end{align}
        The function $\widetilde{\Delta}_t$ can be seen as an approximate worst-case policy. Note that it may not satisfy the constraint $\Delta_t(\cdot) \in \ovl{\Delta}$ and that it is a worst-case with respect to the approximate $Q$-function given by $\bQ_t$ and not the true $Q$-function or the nonlinear dynamics $f$. For this reason, $\widetilde{\Delta}_t$ yields an upper bound on the worst-case performance if used with $\widetilde{Q}_t$, but not necessarily with the true dynamics $f$.
        Plugging \eqref{eq:approxUncertainty} into the Bellman inequality \eqref{eq:nonConvexMinMax} yields
        \begin{align*}
            0 \geq \minimize_{\dvr{\bM_t} \in \calM, \delta_t^u}
            \begin{pmatrix}
            1\\ \delta_t^x\\ \delta_t^u
        \end{pmatrix}^\top
        \left(
        \begin{pmatrix}
            \dvr{\bar{q}_t^{11}} & \dvr{\bar{q}_t^{12}} & \dvr{\bar{q}_t^{13}}\\
            \dvr{\bar{q}_t^{21}} & \dvr{\ovl{\bQ}_t^{22}} & \dvr{\ovl{\bQ}_t^{23}}\\
            \dvr{\bar{q}_t^{31}} & \dvr{\ovl{\bQ}_t^{32}} & \dvr{\ovl{\bQ}_t^{33}}
        \end{pmatrix}
        -
        \begin{pmatrix}
            \dvr{\bar{q}_t^{14}}\\
            \dvr{\ovl{\bQ}_t^{24}}\\
            \dvr{\ovl{\bQ}_t^{34}}
        \end{pmatrix}
        \right.\\
        \resizebox{\linewidth}{!}{$
        \left.
        (\dvr{\ovl{\bQ}_t^{44}})^{-1}
        \begin{pmatrix}
            \dvr{\bar{q}_t^{14}}\\
            \dvr{\ovl{\bQ}_t^{24}}\\
            \dvr{\ovl{\bQ}_t^{34}}
        \end{pmatrix}^\top
        \right)
        \begin{pmatrix}
            1\\ \delta_t^x\\ \delta_t^u
        \end{pmatrix}
        -
        \begin{pmatrix}
                1\\
                \delta_t^x
            \end{pmatrix}^\top
            \begin{pmatrix}
                \dvr{p_t^{11}} & \dvr{p_t^{12}}\\
                \dvr{p_t^{21}} & \dvr{\bP_t^{22}}
            \end{pmatrix} \begin{pmatrix}
                1\\
                \delta_t^x
            \end{pmatrix}$.}
        \end{align*}
        Further, if $\dvr{\ovl{\bQ}_t^{33}} - \dvr{\ovl{\bQ}_t^{34}}(\dvr{\ovl{\bQ}_t^{44}})^{-1}\dvr{\ovl{\bQ}_t^{43}}$ is positive definite, then also $\delta_t^u$ can be computed as
        \begin{align*}
            \delta_t^u = \begin{pmatrix}
                \dvr{k_t^1} & \dvr{\bK_t^2}
            \end{pmatrix}
            \begin{pmatrix}
                1\\ \delta_t^x
            \end{pmatrix}
            = \dvr{\bK_t} \begin{pmatrix}
                1\\ \delta_t^x
            \end{pmatrix}.
        \end{align*}
        
        We finally obtain the requirement that
        \begin{align*}
        \resizebox{\linewidth}{!}{$
        \begin{pmatrix}
                1\\
                \delta_t^x
            \end{pmatrix}^\top
        \left(
        \begin{pmatrix}
            1 & 0\\
            0 & \bI\\
            \dvr{k_t^1} & \dvr{\bK_t^2}
        \end{pmatrix}^\top
        \left(
        \begin{pmatrix}
            \dvr{\bar{q}_t^{11}} & \dvr{\bar{q}_t^{12}} & \dvr{\bar{q}_t^{13}}\\
            \dvr{\bar{q}_t^{21}} & \dvr{\ovl{\bQ}_t^{22}} & \dvr{\ovl{\bQ}_t^{23}}\\
            \dvr{\bar{q}_t^{31}} & \dvr{\ovl{\bQ}_t^{32}} & \dvr{\ovl{\bQ}_t^{33}}
        \end{pmatrix}
        -
        \begin{pmatrix}
            \dvr{\bar{q}_t^{14}}\\
            \dvr{\ovl{\bQ}_t^{24}}\\
            \dvr{\ovl{\bQ}_t^{34}}
        \end{pmatrix}
        \right. \right.$}\\
        \resizebox{\linewidth}{!}{$
        \left. \left.
        (\dvr{\ovl{\bQ}_t^{44}})^{-1}
        \begin{pmatrix}
            \dvr{\bar{q}_t^{14}}\\
            \dvr{\ovl{\bQ}_t^{24}}\\
            \dvr{\ovl{\bQ}_t^{34}}
        \end{pmatrix}^\top
        \right)
        \begin{pmatrix}
            1 & 0\\
            0 & \bI\\
            \dvr{k_t^1} & \dvr{\bK_t^2}
        \end{pmatrix}
        -
            \begin{pmatrix}
                \dvr{p_t^{11}} & \dvr{p_t^{12}}\\
                \dvr{p_t^{21}} & \dvr{\bP_t^{22}}
            \end{pmatrix}
            \right)\begin{pmatrix}
                1\\
                \delta_t^x
            \end{pmatrix}$}
        \end{align*}
        must be smaller or equal to zero for all $\delta_t^x$. This is implied by the matrix inequality constraint
        \begin{align*}
            \begin{pmatrix}
            1 & 0\\
            0 & \bI\\
            \dvr{k_t^1} & \dvr{\bK_t^2}
        \end{pmatrix}^\top
        \left(
        \begin{pmatrix}
            \dvr{\bar{q}_t^{11}} - \dvr{p_t^{11}} & \dvr{\bar{q}_t^{12}}-\dvr{p_t^{12}} & \dvr{\bar{q}_t^{13}}\\
            \dvr{\bar{q}_t^{21}}-\dvr{p_t^{21}} & \dvr{\ovl{\bQ}_t^{22}}-\dvr{\bP_t^{22}} & \dvr{\ovl{\bQ}_t^{23}}\\
            \dvr{\bar{q}_t^{31}} & \dvr{\ovl{\bQ}_t^{32}} & \dvr{\ovl{\bQ}_t^{33}}
        \end{pmatrix} \right.\\
        \left.
        -
        \begin{pmatrix}
            \dvr{\bar{q}_t^{14}}\\
            \dvr{\ovl{\bQ}_t^{24}}\\
            \dvr{\ovl{\bQ}_t^{34}}
        \end{pmatrix}
        (\dvr{\ovl{\bQ}_t^{44}})^{-1}
        \begin{pmatrix}
            \dvr{(\bar{q}_t^{14}}\\
            \dvr{\ovl{\bQ}_t^{24}}\\
            \dvr{\ovl{\bQ}_t^{34}}
        \end{pmatrix}^\top
        \right)
        \begin{pmatrix}
            1 & 0\\
            0 & \bI\\
            \dvr{k_t^1} & \dvr{\bK_t^2}
        \end{pmatrix}\prec 0,
        \end{align*}
        and we further require $\dvr{\ovl{\bQ}_t^{33}} - \dvr{\ovl{\bQ}_t^{34}}(\dvr{\ovl{\bQ}_t^{44}})^{-1}\dvr{\ovl{\bQ}_t^{43}} \succ 0$ and $\dvr{\ovl{\bQ}_t^{44}} \prec 0$ due to the derivation of this inequality. By virtue of the Schur complement, we express the first and the third inequality compactly as
        \begin{align}
        &\resizebox{\linewidth}{!}{$
        (\star)^\top
        \begin{pmatrix}
            \dvr{\bar{q}_t^{11}} - \dvr{p_t^{11}} & \dvr{\bar{q}_t^{12}}-\dvr{p_t^{12}} & \dvr{\bar{q}_t^{13}} & \dvr{\bar{q}_t^{14}}\\
            \dvr{\bar{q}_t^{21}}-\dvr{p_t^{21}} & \dvr{\ovl{\bQ}_t^{22}}-\dvr{\bP_t^{22}} & \dvr{\ovl{\bQ}_t^{23}} & \dvr{\ovl{\bQ}_t^{24}}\\
            \dvr{\bar{q}_t^{31}} & \dvr{\ovl{\bQ}_t^{32}} & \dvr{\ovl{\bQ}_t^{33}} & \dvr{\ovl{\bQ}_t^{34}}\\
            \dvr{\bar{q}_t^{41}} & \dvr{\ovl{\bQ}_t^{42}} & \dvr{\ovl{\bQ}_t^{43}} & \dvr{\ovl{\bQ}_t^{44}}
        \end{pmatrix}
        \begin{pmatrix}
            1 & 0 & 0\\
            0 & \bI & 0\\
            \dvr{k_t^1} & \dvr{\bK_t^2} & 0\\
            0 & 0 & \bI
        \end{pmatrix}$} \nonumber\\
        & \hspace{68mm}\prec 0.\label{eq:backwardMatrixInequality}
        \end{align}

Solving the matrix inequality \eqref{eq:backwardMatrixInequality} is at the center of the backward pass we propose. Due to our derivations, this matrix inequality guarantees that our value function approximations $\widetilde{V}_t(\cdot)$ are upper bounds on the value functions $V_{t,(\tilde{\pi}_t)}(\cdot)$ of the policies we optimize over and, thereby, also upper bounds on the optimal value function. To obtain the best upper bound, the proposed robust DDP procedure includes the minimization of this upper bound, i.e., the optimization of $\dvr{\bP_t}$ for some positive semi-definite matrix $\bSigma_t$ as follows:
\begin{align}
        \minimize_{\dvr{\bM_t} \in \calM,\dvr{\bP_t},\dvr{\bK_t}} ~&~ \trace \bSigma_t \dvr{\bP_t}. \label{eq:backwardMatrixInequality2}\\
        \mathrm{subject~to} ~&~ \eqref{eq:backwardMatrixInequality}\nonumber
\end{align}
This robust DDP method is summarized in Algorithm~\ref{alg:1}.

\begin{algorithm}
    \caption{Robust DDP}\label{alg:1}
    \begin{algorithmic}
        \State \textbf{Require:} $x_0$ and $\widetilde{V}_T(\cdot) \geq V_T(\cdot)$
        \While{$\|(\delta_t^x)\|_{\ell_2} \geq \varepsilon$}
        \State \emph{Backward Pass}
        \For{$t = T-1,\ldots,0$}
        \State Compute $\widetilde{Q}_t(\cdot)$ s.t. $\widetilde{Q}_t(\cdot) \geq  f_0(\cdot) + \widetilde{V}_{t+1}(f(\cdot))$
        \State Solve \eqref{eq:backwardMatrixInequality2} for $\widetilde{V}_t$ and $\tilde{\pi}_t$
        \EndFor
        \State \emph{Forward Pass}
        \For{$t = 0,\ldots,T-1$}
        \State $w_t \gets 0$, $u_t \gets \tilde{\pi}_t(x_t)$
        \State $x_{k+1} \gets f(x_t,u_t,w_t)$
        \State $z_t \gets g(x_t,u_t,w_t)$
        \EndFor
		\EndWhile
	\end{algorithmic}
\end{algorithm}

We stress that, due to the robust treatment of the nonlinearity, no step size control is needed in Algorithm~\ref{alg:1}. However, such measures become relevant when we relax the requirement of robustly treating the nonlinearity (see Section \ref{sec:5}). Policies computed according to Algorithm \ref{alg:1} are robust in the sense of the next theorem.

\begin{theorem}[Robust DDP]
    \label{thm:robustnessStep}
    Assume that $\widetilde{V}_{t+1}(\cdot)$ is an upper bound on $V_{t+1}(\cdot)$ and $\widetilde{Q}_t(\cdot)$ satisfies
    \begin{align}
        \widetilde{Q}_t(x,u,w) \geq f_0(x,u) + \widetilde{V}_{t+1}(f(x,u,w)) \label{eq:upperBounding}
    \end{align}
    for all $x \in \bbR^n, u \in \bbR^m, w \in \bbR^{d}$. If there exist $\dvr{\bM_t} \in \calM, \dvr{\bP_t} \succ 0$ and $\dvr{\bK_t}$ satisfying \eqref{eq:backwardMatrixInequality}, then the function
    \begin{align*}
        \widetilde{V}_t(x) = \begin{pmatrix}
            1\\
            \delta_t^x
        \end{pmatrix}^\top
        \dvr{\bP_t}
        \begin{pmatrix}
            1\\
            \delta_t^x
        \end{pmatrix}, \qquad \delta_t^x = x - x_t
    \end{align*}
    is an upper bound on the true value function $V_t(\cdot)$ of \eqref{eq:origOpt}.
\end{theorem}

\begin{proof}
    Follow the arguments from \eqref{eq:disturbanceConst} to \eqref{eq:backwardMatrixInequality}.
\end{proof}

\begin{corollary}
    \label{cor:robustnessDDP}
    The sequence of approximate value functions $\widetilde{V}_t(\cdot)$ generated by Algorithm \ref{alg:1} is an upper bound for the value function $V_{t,(\tilde{\pi}_t)}(\cdot)$ of the generated policy $\tilde{\pi}_t$ and for the optimal value function $V_t(\cdot)$:
    \begin{align*}
        \widetilde{V}_t(x) &\geq V_{t,(\tilde{\pi}_t)}(x) \geq V_t(x) & \forall x \in \bbR^n, t=0,\ldots,T.
    \end{align*}
\end{corollary}

\section{Convexification of the backward pass}
\label{sec:4}

In this section, we cover the convexification of \eqref{eq:backwardMatrixInequality2}. At this point, we stress that there are many possibilities for convexification and choosing the one suitable for a particular form of $\bQ_t$ and $\dvr{\bM_t}$ is often an art. For this reason, we showcase three possibilities that should give the reader concrete ideas how to approach this problem.

\subsection{A simple case}
\label{sec:4.1}

If in \eqref{eq:backwardMatrixInequality2}, the matrices $\bar{q}_t^{31} , \ovl{\bQ}_t^{32} , \ovl{\bQ}_t^{33}, \ovl{\bQ}_t^{34}$ do not contain decision variables and $\ovl{\bQ}_t^{33}$ is positive definite, then a convexification of this matrix inequality can be achieved by a single Schur complement. The considered independence of $\bar{q}_t^{31} , \ovl{\bQ}_t^{32} , \ovl{\bQ}_t^{33}, \ovl{\bQ}_t^{34}$ of decision variables is often obtained, when in \eqref{eq:disturbanceConst} $g(\cdot)$ is not a function of $u$. 

Consider the matrix $\dvr{\bL_t}$ equal to
\begin{align*}
        \resizebox{\linewidth}{!}{$
    (\star)^\top
        \begin{pmatrix}
            \dvr{\bar{q}_t^{11}} - \dvr{p_t^{11}} & \dvr{\bar{q}_t^{12}}-\dvr{p_t^{12}} & \bar{q}_t^{13} & \dvr{\bar{q}_t^{14}}\\
            \dvr{\bar{q}_t^{21}}-\dvr{p_t^{21}} & \dvr{\ovl{\bQ}_t^{22}}-\dvr{\bP_t^{22}} & \ovl{\bQ}_t^{23} & \dvr{\ovl{\bQ}_t^{24}}\\
            \bar{q}_t^{31} & \ovl{\bQ}_t^{32} & 0 & \ovl{\bQ}_t^{34}\\
            \dvr{\bar{q}_t^{41}} & \dvr{\ovl{\bQ}_t^{42}} & \ovl{\bQ}_t^{43} & \dvr{\ovl{\bQ}_t^{44}}
        \end{pmatrix}
        \begin{pmatrix}
            1 & 0 & 0\\
            0 & \bI & 0\\
            \dvr{k_t^1} & \dvr{\bK_t^2} & 0\\
            0 & 0 & \bI
        \end{pmatrix}.$}
\end{align*}
Then, $\dvr{\bL_t}$ is affine linear in $\dvr{\bK_t}, \dvr{\bP_t}$ and possible 
decision variables in $\dvr{\ovl{\bQ}_t}$. In this case, \eqref{eq:backwardMatrixInequality} can be expressed as
\begin{align*}
    \dvr{\bL_t} + \begin{pmatrix}
        \dvr{k_t^1} & \dvr{\bK_t^2} & 0
    \end{pmatrix}^\top
    \ovl{\bQ}_t^{33}
    \begin{pmatrix}
        \dvr{k_t^1} & \dvr{\bK_t^2} & 0
    \end{pmatrix} \prec 0.
\end{align*}
Due to $\ovl{\bQ}_t^{33} \succ 0$, a Schur complement yields
\begin{align*}
    \begin{pmatrix}
    \dvr{l_t^{11}} & \dvr{l_t^{12}} & \dvr{l_t^{13}} & \dvr{(k_t^{1})^\top}\\
    \dvr{l_t^{21}} & \dvr{\bL_t^{22}} & \dvr{\bL_t^{23}} & \dvr{(\bK_t^{2})^\top}\\
    \dvr{l_t^{31}} & \dvr{\bL_t^{32}} & \dvr{\bL_t^{33}} & 0\\
    \dvr{k_t^1} & \dvr{\bK_t^2} & 0 & -(\ovl{\bQ}_t^{33})^{-1}
    \end{pmatrix} \prec 0.
\end{align*}
This matrix inequality is linear in all decision variables.

\subsection{Dualization for convexification}
\label{sec:4.2}

In general, we cannot assume that subblocks of $\dvr{\ovl{\bQ}_t}$ do not depend on decision variables. Nevertheless, convexification can be performed for a large class of $\dvr{\ovl{\bQ}_t}$. This class contains, for example, the case where $\bQ_t$ is fixed and $\dvr{\bM_t}$ satisfies \eqref{eq:genericMultipliers}. For dual design, we split again the $Q$-function matrix $\dvr{\ovl{\bQ}_t}$ into $\bQ_t$ and the multiplier $\dvr{\bM_t}$ to obtain from \eqref{eq:backwardMatrixInequality2} the problem
        \begin{align*}
            \minimize_{\dvr{\bP_t} \succ 0,\dvr{\bK_t},\dvr{\bM_t} \in \calM} ~&~ \trace \dvr{\bP_t} \bSigma_t
        \end{align*}
        subject to
        \begin{align}
            \resizebox{\linewidth}{!}{$
            (\star)^\top
            \begin{pmatrix}
                -\dvr{p_t^{11}} & - \dvr{p_t^{12}}\\
                -\dvr{p_t^{21}} & -\dvr{\bP_t^{22}}\\
                 & & \dvr{\bM_t^{44}} & \dvr{\bM_t^{43}} & \dvr{\bM_t^{42}} & \dvr{m_t^{41}}\\
                 & & \dvr{\bM_t^{34}} & \dvr{\bM_t^{33}} & \dvr{\bM_t^{32}} & \dvr{m_t^{31}}\\
                 & & \dvr{\bM_t^{24}} & \dvr{\bM_t^{24}} & \dvr{\bM_t^{22}} & \dvr{m_t^{21}}\\
                 & & \dvr{m_t^{14}} & \dvr{m_t^{13}} & \dvr{m_t^{12}} & \dvr{m_t^{11}}\\
                 & & & & & & q_t^{11} & q_t^{12} & q_t^{13} & q_t^{14}\\
                 & & & & & & q_t^{21} & \bQ_t^{22} & \bQ_t^{23} & \bQ_t^{24}\\
                 & & & & & & q_t^{31} & \bQ_t^{32} & \bQ_t^{33} & \bQ_t^{34}\\
                 & & & & & & q_t^{41} & \bQ_t^{42} & \bQ_t^{43} & \bQ_t^{44}
            \end{pmatrix}
            \begin{pmatrix}
                1 & 0 & 0\\
                0 & \bI & 0\\
                0 & 0 & \bI\\
                \dvr{k_t^1} & \dvr{\bK_t^2} & 0\\
                0 & \bI & 0\\
                1 & 0 & 0\\
                1 & 0 & 0\\
                0 & \bI & 0\\
                \dvr{k_t^1} & \dvr{\bK_t^2} & 0\\
                0 & 0 & \bI
            \end{pmatrix}\prec 0.$}
            \label{eq:primalMI}
        \end{align}
        
        The key is to apply the following lemma to derive an equivalent form of the matrix inequality constraint.
        
        \begin{lemma}[derived from Lemma 10.2, \cite{scherer2000robust}]
        \label{lem:dualizationLemma}
            Let $\bP$ be symmetric. Then the \emph{primal matrix inequalities}
            \begin{align*}
            \begin{pmatrix}
                \bI\\
                \bW(\bK)
            \end{pmatrix}^\top
            \bP
            \begin{pmatrix}
                \bI\\
                \bW(\bK)
            \end{pmatrix}
            \prec 0,
            ~~~~~~~
            \begin{pmatrix}
                0\\
                \bI
            \end{pmatrix}^\top
            \bP
            \begin{pmatrix}
                0\\
                \bI
            \end{pmatrix}\succ 0.
        \end{align*}
        are equivalent to the \emph{dual matrix inequalities}
        \begin{align*}
        \resizebox{\linewidth}{!}{$
            \begin{pmatrix}
                \bW(\bK)^\top\\
                -\bI
            \end{pmatrix}^\top
            \bP^{-1}
            \begin{pmatrix}
                \bW(\bK)^\top\\
                -\bI
            \end{pmatrix}
            \succ 0,
            ~~
            \begin{pmatrix}
                \bI\\
                0
            \end{pmatrix}^\top
            \bP^{-1}
            \begin{pmatrix}
                \bI\\
                0
            \end{pmatrix}\prec 0.$}
        \end{align*}
        \end{lemma}

        To denote the dual matrix inequality of \eqref{eq:primalMI}, we define
        \begin{align*}
            &\begin{pmatrix}
                 \dvr{p_t^{11}} &  \dvr{p_t^{12}}\\
                 \dvr{p_t^{21}} & \dvr{\bP_t^{22}}
            \end{pmatrix}^{-1}
            =
            \dvr{\bP_t^{-1}}
            =
            \dvr{\widetilde{\bP}_t}
            =
            \begin{pmatrix}
                \dvr{\tilde{p}_t^{11}} &  \dvr{\tilde{p}_t^{12}}\\
                \dvr{\tilde{p}_t^{21}} & \dvr{\widetilde{\bP}_t^{22}}
            \end{pmatrix},\\
        &\resizebox{\linewidth}{!}{$
            \begin{pmatrix}
                \dvr{\bM_t^{44}} & \dvr{\bM_t^{43}} & \dvr{\bM_t^{42}} & \dvr{m_t^{41}}\\
                \dvr{\bM_t^{34}} & \dvr{\bM_t^{33}} & \dvr{\bM_t^{32}} & \dvr{m_t^{31}}\\
                \dvr{\bM_t^{24}} & \dvr{\bM_t^{23}} & \dvr{\bM_t^{22}} & \dvr{m_t^{21}}\\
                \dvr{m_t^{14}} & \dvr{m_t^{13}} & \dvr{m_t^{12}} & \dvr{m_t^{11}}\\
            \end{pmatrix}^{-1}
            =
            \dvr{\bM}_t^{-1}
            =
            \dvr{\widetilde{\bM}_t}
            =
            \begin{pmatrix}
                \dvr{\widetilde{\bM}_t^{44}} & \dvr{\widetilde{\bM}_t^{43}} & \dvr{\widetilde{\bM}_t^{42}} & \dvr{\tilde{m}_t^{41}}\\
                \dvr{\widetilde{\bM}_t^{34}} & \dvr{\widetilde{\bM}_t^{33}} & \dvr{\widetilde{\bM}_t^{32}} & \dvr{\tilde{m}_t^{31}}\\
                \dvr{\widetilde{\bM}_t^{24}} & \dvr{\widetilde{\bM}_t^{23}} & \dvr{\widetilde{\bM}_t^{22}} & \dvr{\tilde{m}_t^{21}}\\
                \dvr{\tilde{m}_t^{14}} & \dvr{\tilde{m}_t^{13}} & \dvr{\tilde{m}_t^{12}} & \dvr{\tilde{m}_t^{11}}
            \end{pmatrix},$}\\
        &\resizebox{\linewidth}{!}{$
            \begin{pmatrix}
                q_t^{11} & q_t^{12} & q_t^{13} & q_t^{14}\\
                q_t^{21} & \bQ_t^{22} & \bQ_t^{23} & \bQ_t^{24}\\
                q_t^{31} & \bQ_t^{32} & \bQ_t^{33} & \bQ_t^{34}\\
                q_t^{41} & \bQ_t^{42} & \bQ_t^{43} & \bQ_t^{44}
            \end{pmatrix}^{-1}
            =
            \bQ_t^{-1}
            =
            \widetilde{\bQ}_t
            =
            \begin{pmatrix}
                \tilde{q}_t^{11} & \tilde{q}_t^{12} & \tilde{q}_t^{13} & \tilde{q}_t^{14}\\
                \tilde{q}_t^{21} & \widetilde{\bQ}_t^{22} & \widetilde{\bQ}_t^{23} & \widetilde{\bQ}_t^{24}\\
                \tilde{q}_t^{31} & \widetilde{\bQ}_t^{32} & \widetilde{\bQ}_t^{33} & \widetilde{\bQ}_t^{34}\\
                \tilde{q}_t^{41} & \widetilde{\bQ}_t^{42} & \widetilde{\bQ}_t^{43} & \widetilde{\bQ}_t^{44}
            \end{pmatrix}.$}
        \end{align*}

        In addition, for dualization, we have to require
        \begin{align}
            \begin{pmatrix}
                \dvr{\bM_t^{33}} & \dvr{\bM_t^{32}} & \dvr{m_t^{31}}\\
                \dvr{\bM_t^{23}} & \dvr{\bM_t^{22}} & \dvr{m_t^{21}}\\
                \dvr{m_t^{13}} & \dvr{m_t^{12}} & \dvr{m_t^{11}}\\
                 & & & q_t^{11} & q_t^{12} & q_t^{13} & q_t^{14}\\
                 & & & q_t^{21} & \bQ_t^{22} & \bQ_t^{23} & \bQ_t^{24}\\
                 & & & q_t^{31} & \bQ_t^{32} & \bQ_t^{33} & \bQ_t^{34}\\
                 & & & q_t^{41} & \bQ_t^{42} & \bQ_t^{43} & \bQ_t^{44}
            \end{pmatrix} \succ 0.
            \label{eq:regularity}
        \end{align}
        By Lemma \ref{lem:dualizationLemma}, we then obtain that the matrix
        \begin{align*}
            \resizebox{\linewidth}{!}{$
            (\star)^\top
            \begin{pmatrix}
                -\dvr{\tilde{p}_t^{11}} & - \dvr{\tilde{p}_t^{12}}\\
                -\dvr{\tilde{p}_t^{21}} & -\dvr{\widetilde{\bP}_t^{22}}\\
                 & & \dvr{\widetilde{\bM}_t^{44}} & \dvr{\widetilde{\bM}_t^{43}} & \dvr{\widetilde{\bM}_t^{42}} & \dvr{\tilde{m}_t^{41}}\\
                 & & \dvr{\widetilde{\bM}_t^{34}} & \dvr{\widetilde{\bM}_t^{33}} & \dvr{\widetilde{\bM}_t^{32}} & \dvr{\tilde{m}_t^{31}}\\
                 & & \dvr{\widetilde{\bM}_t^{24}} & \dvr{\widetilde{\bM}_t^{23}} & \dvr{\widetilde{\bM}_t^{22}} & \dvr{\tilde{m}_t^{21}}\\
                 & & \dvr{\tilde{m}_t^{14}} & \dvr{\tilde{m}_t^{13}} & \dvr{\tilde{m}_t^{12}} & \dvr{\tilde{m}_t^{11}}\\
                 & & & & & & \tilde{q}_t^{11} & \tilde{q}_t^{12} & \tilde{q}_t^{13} & \tilde{q}_t^{14}\\
                 & & & & & & \tilde{q}_t^{21} & \widetilde{\bQ}_t^{22} & \widetilde{\bQ}_t^{23} & \widetilde{\bQ}_t^{24}\\
                 & & & & & & \tilde{q}_t^{31} & \widetilde{\bQ}_t^{32} & \widetilde{\bQ}_t^{33} & \widetilde{\bQ}_t^{34}\\
                 & & & & & & \tilde{q}_t^{41} & \widetilde{\bQ}_t^{42} & \widetilde{\bQ}_t^{43} & \widetilde{\bQ}_t^{44}
            \end{pmatrix}
	        \begin{pmatrix}
	        \dvr{(k_t^1)^\top} & 0 & 1 & 1 & 0 & \dvr{(k_t^1)^\top} & 0\\
	        \dvr{(\bK_t^2)^\top} & \bI & 0 & 0 & \bI &\dvr{(\bK_t^2)^\top} & 0\\
	        0 & 0 & 0 & 0 & 0 & 0 & \bI\\
	        -\bI & 0 & 0 & 0 & 0 & 0 & 0\\
	        0 & -\bI & 0 & 0 & 0 & 0 & 0\\
	        0 & 0 & -1 & 0 & 0 & 0 & 0\\
	        0 & 0 & 0 & -1 & 0 & 0 & 0\\
	        0 & 0 & 0 & 0 & -\bI & 0 & 0\\
	        0 & 0 & 0 & 0 & 0 & -\bI & 0\\
	        0 & 0 & 0 & 0 & 0 & 0 & -\bI
	    \end{pmatrix}$}
        \end{align*}
must be postive definite, together with the requirements
\begin{align*}
    \begin{pmatrix}
        \dvr{\tilde{p}_t^{11}} & \dvr{\tilde{p}_t^{12}}\\
        \dvr{\tilde{p}_t^{21}} & \dvr{\widetilde{\bP}_t^{22}}
    \end{pmatrix}
    \succ 0, \qquad
    \dvr{\widetilde{\bM}_t^{44}} \prec 0
\end{align*}
as an equivalent of \eqref{eq:primalMI}. Since $\dvr{\bP_t} \succ 0$ is required, convexity is obtained via the Schur complement.

One apparent disadvantage of this approach is \eqref{eq:regularity}. The assumption that e.g. $\bQ_t$ is positive definite seems overly restrictive as, for example, $\bQ_t$ is usually only positive semi-definite for linear systems. Therefore, we mention that \eqref{eq:regularity} can be relaxed to positive semi-definiteness by making use of systematic perturbations.

\subsection{The canonical dualization}
\label{sec:4.3}

The dualization approach from Section \ref{sec:4.2} convexifies the robust dynamic programming step \eqref{eq:backwardMatrixInequality2} for arbitrary affine linear dependence of $\dvr{\bM_t}$ and $\bQ_t$ on decision variables. However, it in turn requires a convex parametrization of the inverse matrices $\dvr{\widetilde{\bM}_t}$ and $\widetilde{\bQ}_t$, which might not always be achievable. For this reason, we explore a general approach for dualization in this section. This general result is based on a factorization approach for $\dvr{\ovl{\bQ}_t}$ that allows for efficient parametrization of the inverses.
Assume that $\dvr{\ovl{\bQ}_t}$ can be written in the form
\begin{align*}
    \dvr{\ovl{\bQ}_t} = 
    \begin{pmatrix}
    \bPi_{\ovl{\bQ}}^{0}\\
    \bPi_{\ovl{\bQ}}^{1}\\
    \bPi_{\ovl{\bQ}}^{2}
    \end{pmatrix}^\top
    \begin{pmatrix}
    \bS_0\\
    & \dvr{\bS_1}\\
    & & - \dvr{\bS_2}
    \end{pmatrix}
    \begin{pmatrix}
    \bPi_{\ovl{\bQ}}^{0}\\
    \bPi_{\ovl{\bQ}}^{1}\\
    \bPi_{\ovl{\bQ}}^{2}
    \end{pmatrix}
\end{align*}
where $\bS_0$ is a constant matrix and $\dvr{\bS_1}$, $\dvr{\bS_2}$ have positive definite matrix decision variables on their diagonals, i.e.,
\begin{align*}
	\resizebox{\linewidth}{!}{$
    \dvr{\bS_1} = \diag \left(
        \dvr{\bLambda_1}, \dvr{\bLambda_2}, \ldots,\dvr{\bLambda_{s_1}} \right), ~
    \dvr{\bS_2} = \diag \left(\dvr{\bGamma_1}, \dvr{\bGamma_2},\ldots \dvr{\bGamma_{s_2}}\right),$}
\end{align*}
where repeated blocks are explicitly allowed. An example for this structure of $\dvr{\ovl{\bQ}_t} = \bQ_t + \dvr{\bM}_t$ is obtained if $\bQ_t$ is constant and $\dvr{\bM_t}$ is of the form \eqref{eq:genericMultipliers}. In this case, we have
\begin{align*}
    \dvr{\ovl{\bQ}_t} = \bQ_t + \sum_{i=1}^s \dvr{\lambda_t^{(i)}}\bM^{(i)},
\end{align*}
where $\bM^{(i)}$ are symmetric matrices and $\dvr{\lambda_t^{(i)}} \geq 0$ are non-negative parameters. Here, the $\bM^{(i)}$ can be decomposed with full rank matrices $\bM_p^{(i)}$ and $\bM_m^{(i)}$ according to
\begin{align*}
    \bM^{(i)} = (\bM_p^{(i)})^\top \bM_p^{(i)} - (\bM_m^{(i)})^\top \bM_m^{(i)} .
\end{align*}
With this decomposition, we can now denote $\dvr{\ovl{\bQ}_t}$ as
\begin{align*}
    \resizebox{\linewidth}{!}{$
    \begin{pmatrix}
        \bI\\
        \bM_p^{(1)}\\
        \vdots\\
        \bM_p^{(s)}\\
        \bM_m^{(1)}\\
        \vdots\\
        \bM_m^{(s)}\\
    \end{pmatrix}^\top
    \begin{pmatrix}
        \bQ_t\\
        & \dvr{\lambda_1} \bI\\
        & & \ddots \\
        & & &\dvr{\lambda_s} \bI\\
        & & & & -\dvr{\lambda_1} \bI\\
        & & & & & \ddots\\
        & & & & & & -\dvr{\lambda_s} \bI
    \end{pmatrix}
    \begin{pmatrix}
        \bI\\
        \bM_p^{(1)}\\
        \vdots\\
        \bM_p^{(s)}\\
        \bM_m^{(1)}\\
        \vdots\\
        \bM_m^{(s)}\\
    \end{pmatrix}$}
\end{align*}
where $\dvr{\lambda_1} \bI$ should be interpreted as repeated $1\times 1$ blocks. 
With our structural assumption, \eqref{eq:backwardMatrixInequality} can be denoted as
\begin{align}
	\resizebox{\linewidth}{!}{$
    (\star)^\top
    \begin{pmatrix}
    -\dvr{\bP_t}\\
    & -\dvr{\bS_2}\\
    & & \dvr{\bS_1}\\
    & & & \bS_0
    \end{pmatrix}
    \begin{pmatrix}
    \bI & 0\\
    \dvr{\bW_{11}(\bK_t)} & \bW_{12}\\
    \dvr{\bW_{21}(\bK_t)} & \bW_{22}\\
    \dvr{\bW_{31}(\bK_t)} & \bW_{32}
    \end{pmatrix}
    \prec 0.$}
    \label{eq:canonicalPrimal}
\end{align}
Here, the matrices $\bW_{ij}$ are defined as
\begin{align*}
    \begin{pmatrix}
        \dvr{\bW_{11}(\bK_t)}\\
        \dvr{\bW_{21}(\bK_t)}\\
        \dvr{\bW_{31}(\bK_t)}
    \end{pmatrix}
    &=
    \begin{pmatrix}
    \bPi_{\ovl{\bQ}}^{2}\\
    \bPi_{\ovl{\bQ}}^{1}\\
    \bPi_{\ovl{\bQ}}^{0}
    \end{pmatrix}
    \begin{pmatrix}
        1 & 0\\
        0 & \bI\\
        0 & 0\\
        \dvr{k_t^1} & \dvr{\bK_t^2}
    \end{pmatrix},\\
    \begin{pmatrix}
        \bW_{12}\\
        \bW_{22}\\
        \bW_{32}
    \end{pmatrix}
    &=
    \begin{pmatrix}
    \bPi_{\ovl{\bQ}}^{2}\\
    \bPi_{\ovl{\bQ}}^{1}\\
    \bPi_{\ovl{\bQ}}^{0}
    \end{pmatrix}
    \begin{pmatrix}
        0\\
        0\\
        \bI\\
        0
    \end{pmatrix}.
\end{align*}
Again, we want to apply the dualization lemma. Unfortunately, the identity block in the projection matrix of \eqref{eq:canonicalPrimal} is not large enough to apply Lemma \ref{lem:dualizationLemma}. Therefore, we need the following modified dualization lemma.

\begin{lemma}
    \label{lem:Dualization2}
    Let $\bP$ be a symmetric matrix and let $\bW_1$ have full column rank with left inverse $\bW_1^\dagger$. Then, the dual matrix inequalities
    \begin{align}
        \resizebox{\linewidth}{!}{$
            \begin{pmatrix}
                (\bW_2\bW_1^\dagger)^\top\\
                -\bI
            \end{pmatrix}^\top
            \bP^{-1}
            \begin{pmatrix}
                (\bW_2\bW_1^\dagger)^\top\\
                -\bI
            \end{pmatrix} \succ 0, ~ \begin{pmatrix}
                \bI\\
                0
            \end{pmatrix}^\top
            \bP^{-1}
            \begin{pmatrix}
                \bI\\
                0
            \end{pmatrix} \prec 0$} \nonumber 
    \end{align}
    imply the primal matrix inequalities
    \begin{align}
        \begin{pmatrix}
                \bW_1\\
                \bW_2
            \end{pmatrix}^\top
            \bP
            \begin{pmatrix}
                \bW_1\\
                \bW_2
            \end{pmatrix} \prec 0, 
            \qquad
            \begin{pmatrix}
                0\\ \bI
            \end{pmatrix}^\top
            \bP
            \begin{pmatrix}
                0\\
                \bI
            \end{pmatrix} \succ 0. \label{eq:calPrimal}
    \end{align}
\end{lemma}
\begin{proof}
    First, we apply Lemma \ref{lem:dualizationLemma} to conclude that the dual matrix inequalities are equivalent to
    \begin{align}
    	\resizebox{\linewidth}{!}{$
        \begin{pmatrix}
            \bI\\
            (\bW_2\bW_1^\dagger)
        \end{pmatrix}^\top
        \bP
        \begin{pmatrix}
            \bI\\
            (\bW_2\bW_1^\dagger)
        \end{pmatrix} \prec 0, \quad 
            \begin{pmatrix}
                0\\ \bI
            \end{pmatrix}^\top
            \bP
            \begin{pmatrix}
                0\\
                \bI
            \end{pmatrix} \succ 0.$} \label{eq:Dualization2_1}
    \end{align}
    Now, since $\bW_1^\dagger$ is a left inverse of $\bW_1$, it holds that $\bW_1^\dagger \bW_1 = \bI$. Consequently, by multiplying \eqref{eq:Dualization2_1} from both sides with $\bW_1$, we obtain \eqref{eq:calPrimal}.
\end{proof}
This form of the dualization lemma can be applied to \eqref{eq:canonicalPrimal}, since it does not require a large identity block in the projection matrix. Unfortunately, here, the primal and the dual matrix inequalities are no longer equivalent, but the dual inequality is a stronger condition than the primal inequality. We apply Lemma \ref{lem:Dualization2} to \eqref{eq:canonicalPrimal} with
\begin{align*}
        \resizebox{\linewidth}{!}{$
    \dvr{\bW_1} = \begin{pmatrix}
    \bI & 0\\
    \dvr{\bW_{11}(\bK_t)} & \bW_{12}
    \end{pmatrix}, ~~~
    \dvr{\bW_2} = \begin{pmatrix}
    \dvr{\bW_{21}(\bK_t)} & \bW_{22}\\
    \dvr{\bW_{31}(\bK_t)} & \bW_{32}
    \end{pmatrix}.$}
\end{align*}
To this end, we require $\dvr{\bW_1}$ to have full column rank implying that also $\bW_{12}$ must have full column rank. Therefore, a left inverse $\bW_{12}^\dagger$ of $\bW_{12}$ exists and
\begin{align*}
    \begin{pmatrix}
    \bI & 0\\
    -\bW_{12}^\dagger\dvr{\bW_{11}(\bK_t)} & \bW_{12}^\dagger
    \end{pmatrix}
    =
    \begin{pmatrix}
    \bI & 0\\
    \dvr{\bW_{11}(\bK_t)} & \bW_{12}
    \end{pmatrix}^\dagger
\end{align*}
is a left inverse $\dvr{\bW_1^\dagger}$ of $\dvr{\bW}_1$. The product $\dvr{\bW_2}\dvr{\bW_1^\dagger}$ equals
\begin{align*}
    \begin{pmatrix}
    \dvr{\bW_{21}(\bK_t)} - \bW_{22}\bW_{12}^\dagger \dvr{\bW_{11}(\bK_t)} & \bW_{22}\bW_{12}^\dagger\\
    \dvr{\bW_{31}(\bK_t)} - \bW_{32}\bW_{12}^\dagger \dvr{\bW_{11}(\bK_t)} & \bW_{32}\bW_{12}^\dagger
    \end{pmatrix}\\
    =:
    \begin{pmatrix}
    \dvr{\widetilde{\bW}_{11}(\bK_t)} & \widetilde{\bW}_{12}\\
    \dvr{\widetilde{\bW}_{21}(\bK_t)} & \widetilde{\bW}_{22}
    \end{pmatrix}
\end{align*}
and with that, we obtain the dual matrix inequality
\begin{align*}
    \resizebox{\linewidth}{!}{$
    (\star)^\top
    \begin{pmatrix}
    -\dvr{\widetilde{\bP}_t}\\
    & -\dvr{\widetilde{\bS}_2}\\
    & & \dvr{\widetilde{\bS}_1}\\
    & & & \widetilde{\bS}_0
    \end{pmatrix}\nonumber
    \begin{pmatrix}
    	\dvr{\widetilde{\bW}_{11}^\top(\bK_t)} & \widetilde{\bW}_{12}^\top\\
    	\dvr{\widetilde{\bW}_{21}^\top (\bK_t)} & \widetilde{\bW}_{22}^\top\\
    	-\bI & 0\\
    	0 & -\bI
    \end{pmatrix}\succ 0.$}
\end{align*}
By a Schur complement w.r.t. $\dvr{\widetilde{\bP}_t}$, we finally obtain
\begin{align*}
        \resizebox{\linewidth}{!}{$
    \begin{pmatrix}
    \dvr{\widetilde{\bS}_1} - \widetilde{\bW}_{12}\dvr{\widetilde{\bS}_2}\widetilde{\bW}_{12}^\top  & -\widetilde{\bW}_{12}\dvr{\widetilde{\bS}_2}\widetilde{\bW}_{22}^\top & \dvr{\widetilde{\bW}_{11}(\bK_t)}\\
    -\widetilde{\bW}_{22}\dvr{\widetilde{\bS}_2}\widetilde{\bW}_{12}^\top & \widetilde{\bS}_0 - \widetilde{\bW}_{22}\dvr{\widetilde{\bS}_2}\widetilde{\bW}_{22}^\top & \dvr{\widetilde{\bW}_{12}(\bK_t)}\\
    \dvr{\widetilde{\bW}_{11}^\top(\bK_t)} & \dvr{\widetilde{\bW}_{12}^\top (\bK_t)} & \dvr{\bP_t}
    \end{pmatrix}\succ 0.$}
\end{align*}
Notice that $\dvr{\widetilde{\bW}_{11}(\bK_t)}$ and $\dvr{\widetilde{\bW}_{21}(\bK_t)}$ depend linearly on $\dvr{\bK_t}$ and that the matrices $\dvr{\widetilde{\bS}_1}$ and $\dvr{\widetilde{\bS}_2}$ can, due to their structure, be convexly parametrized just like $\dvr{\bS_1}$ and $\dvr{\bS_2}$.

Finally, let us mention that none of the convexifications is superior to the others. Thus, Subsection \ref{sec:4.1} is preferable when it is applicable because of its simplicity. If it fails, we should try Subsection \ref{sec:4.2}. Only if the convex parametrization of $\widetilde{\bM}_t$ fails do we recommend the abstract, general convexification from Subsection \ref{sec:4.3}, due to its additional conservatism.

\section{Approximations for the $Q$-function}
\label{sec:5}


The results in Theorem \ref{thm:robustnessStep} and Corollary \ref{cor:robustnessDDP} that robust DDP indeed optimizes the robust performance of the policy $(\tilde{\pi}_t)$ rely on the inequality \eqref{eq:upperBounding} for all $x,u,w$ and $t = T-1,\ldots, 0$. Here, $\widetilde{Q}_t(\cdot)$ is a quadratic function of the form \eqref{eq:quadraticQ} defined by a matrix $\bQ_t$. The function $\widetilde{V}_{t+1}(\cdot)$ is always going to be quadratic. However, nonlinear dynamics $f$ or costs $f_0$ might occur and, in this case, finding a tight robust estimate for $\bQ_t$ can be hard. Non-tight estimates, on the other hand, might not be sufficient for achieving satisfying performance. Even tight estimates might be dissatisfying. For this reason, we discuss estimations for $\bQ_t$ in the course of this section.

\subsection{Approximation by Taylor Expansion}
\label{sec:5.1}

One idea for estimating $\bQ_t$ is to use a Taylor series expansion of order two, i.e., we choose $\bQ_t$ such that 
\begin{align*}
	\resizebox{\linewidth}{!}{$
    f_0(x,u) + \widetilde{V}_{t+1}(f(x,u,w)) = \begin{pmatrix}
        1 & (\delta_t^x)^\top& (\delta_t^w)^\top & (\delta_t^u)^\top
    \end{pmatrix}$}\\
	\resizebox{\linewidth}{!}{$
    \bQ_t
    \begin{pmatrix}
        1 & (\delta_t^x)^\top & (\delta_t^w)^\top & (\delta_t^u)^\top
    \end{pmatrix}^\top
    +
    o(\|\delta_t^x\|^2 + \|\delta_t^w\|^2 + \|\delta_t^u\|^2).$}
\end{align*}
This approximation can easily be computed. Unfortunately, it has the following disadvantages:
\begin{compactitem}
    \item It can violate the robustness requirement \eqref{eq:upperBounding}.
    \item $\bQ_t$ is not guaranteed to be positive semi-definite.
\end{compactitem}
We stress that a mild violation of the requirement \eqref{eq:upperBounding} could be tolerated since this constraint could make the DDP quite conservative. The positive definiteness of $\bQ_t$, on the other hand, is needed for the procedure. We therefore recommend a regularization for $\bQ_t$, i.e., adding a positive definite term onto $\bQ_t$ that ensures that $\bQ_t$ is positive definite, and improves the robustness.

\subsection{Approximation by affine linearization}
\label{sec:5.2}

Instead of approximating $\bQ_t$, we can use a linearization of the dynamics $f(x,u,w)$, i.e.,
\begin{align*}
    f(x,u,w) \approx f_t + \bA_t \delta_t^x + \bB_t^u \delta_t^u + \bB_t^{w} \delta_t^w
\end{align*}
and a quadratic approximation of the function $f_0$, i.e.,
\begin{align*}
    f_0(x,u) \approx \begin{pmatrix}
        1\\
        \delta_t^x\\
        \delta_t^u
    \end{pmatrix}^\top
    \begin{pmatrix}
    r_t^{11} & r_t^{12} & r_t^{13}\\
    r_t^{21} & \bR_t^{22} & \bR_t^{23}\\
    r_t^{31} & \bR_t^{32} & \bR_t^{33}
    \end{pmatrix}
    \begin{pmatrix}
        1\\
        \delta_t^x\\
        \delta_t^u
    \end{pmatrix}.
\end{align*}
For this approximation, we can compute $\bQ_t$ exactly as
\begin{align*}
    \resizebox{\linewidth}{!}{$
    (\star)^\top
    \begin{pmatrix}
    p_{t+1}^{11} & p_{t+1}^{12}\\
    p_{t+1}^{21} & \bP_{t+1}^{22}\\
    & & r_t^{11} & r_t^{12} & r_t^{13}\\
    & & r_t^{21} & \bR_t^{22} & \bR_t^{23}\\
    & &r_t^{31} & \bR_t^{32} & \bR_t^{33}
    \end{pmatrix}
    \begin{pmatrix}
        1 & 0 & 0 & 0\\
        f_t & \bA_t  & \bB_t^u & \bB_t^{w}\\
        1 & 0 & 0 & 0\\
        0 & \bI & 0 & 0\\
        0 & 0 & \bI & 0
    \end{pmatrix}.$}
\end{align*}
This approximation offers the advantage that $\bQ_t$ is always positive semi-definite if $\bR_t$ is positive semi-definite. In addition, we see that for affine linear dynamics and quadratic costs, the matrix $\bQ_t$ can be computed exactly.

\section{Numerical Examples}
\label{sec:6}

To demonstrate the robust performance of our method, we study the inverted pendulum dynamics
\begin{align}
    \begin{split}
    \dot{\theta}(t) &= \omega (t)\\
    \dot{\omega}(t) &= -d_1 \omega (t) + \frac{g}{l}\sin (\theta (t)) + \cos(\theta (t)) \dot{v}(t)\\
    \dot{s}(t) &= v(t)\\
    \dot{v}(t) &= - d_2 v(t) + u(t) .
    \end{split}
    \label{eq:inv_pendulum}
\end{align}
In this model, $s(t)$ is the position of a cart at time $t$ with a pendulum placed on top of it. Accordingly, $v(t)$ is the velocity of the cart, whereas $\theta(t)$ and $\omega(t)$ are the angle and the angular velocity of the pendulum. The input $u(t)$ affects the acceleration of the cart directly. In \eqref{eq:inv_pendulum}, $g = 9.81 \frac{\mathrm{m}}{\mathrm{s}^2}$ is the gravitational acceleration and $l = 1 \mathrm{m}$. The friction parameters $d_1,d_2 \in [0,0.1]\frac{1}{s}$ are assumed to be uncertain. The control objective is to minimize
\begin{align}
    \int_{0}^{10} u^2(t) \diff t + 1000 \|x(10)\|^2, \label{eq:objective_MonteCarlo}
\end{align}
where $x(t) = 0$ corresponds to the upright pendulum position.
The problem is discretized by partitioning the time interval $[0,10]$ into $T = 50$ intervals $[t_k,t_{k+1}]$ of the same size and choosing $u(t)$ to be equal to the constant value $u_k$ on the intervals. The discrete state is then chosen as $x_k = x(t_k)$ and the discrete dynamics are computed using an ODE solver. To enable the application of robust DDP to the considered problem, the uncertainties $d_1, d_2 \in [0,0.1]$ are covered with a generalized plant using the standard modelling steps from robust control.

To demonstrate the benefits of robust DDP, we consider a Monte Carlo series of simulations with system \eqref{eq:inv_pendulum} and parameters $d_1,d_2$ sampled from the uniform distribution on $[0,0.1]\frac{1}{s}$. In addition to $d_1$ and $d_2$ also the initial conditions $\theta (0) \in [\pi-1,\pi+1], \omega (0) \in [-0.5,0.5], s(0) \in [-1,1], v(0) \in [-1,1]$ are sampled from uniform distributions. Our robust DDP method is then applied, where the $Q$-function is approximated using an affine linearization of the system dynamics (see Subsection \ref{sec:5.2}) and the third possibility for convexification (see Subsection \ref{sec:4.3}) is applied. For comparison of the performance of robust DDP and nominal DDP, we generate a single policy $\tilde{\pi}_t (x_t) = k_t^1 + \bK_t^2 \delta_t^x$ with either robust or nominal DDP, run a simulation of system \eqref{eq:inv_pendulum} and measure the performance \eqref{eq:objective_MonteCarlo} in each Monte Carlo shot. The running average costs of the Monte Carlo shots of robust and nominal DDP are shown in Figure~\ref{fig:RunningAverages}. As we can see, the costs resulting from robust DDP are significantly lower than the costs from nominal DDP. Particularly, we highlight the outliers in the costs of nominal DDP in Figure~\ref{fig:RunningAverages}, where the nominal planning method fails to stabilize the pendulum in the upright position. This does not happen with robust DDP.

\begin{figure}
	\vspace{1mm}
    \centering
    \includegraphics[width = \linewidth]{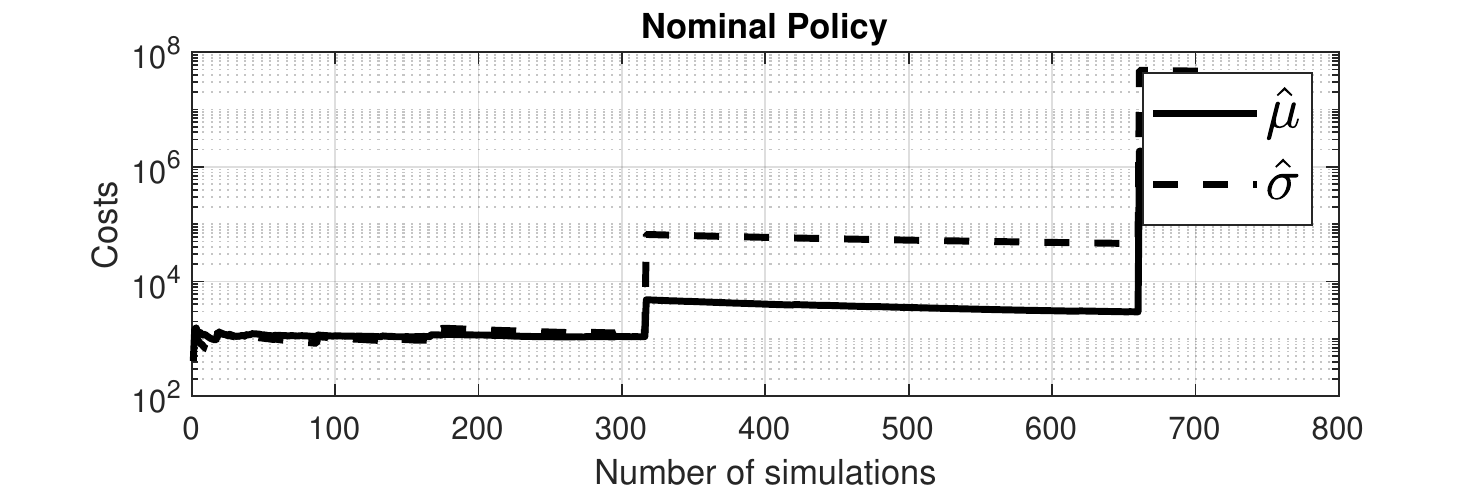}
    \includegraphics[width = \linewidth]{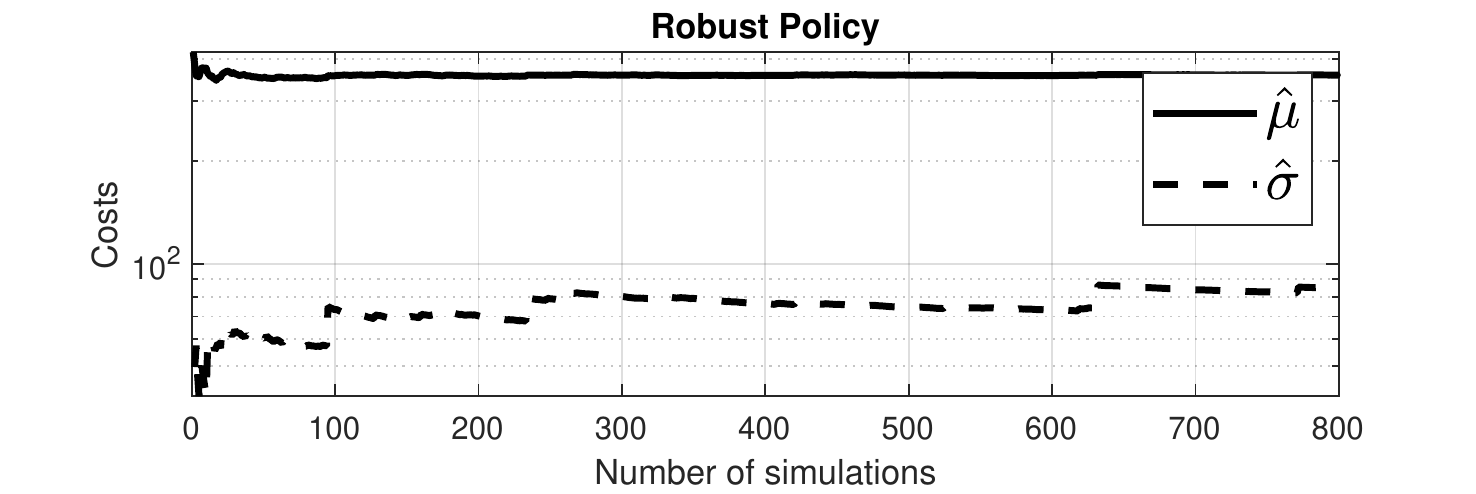}
    \vspace{-7mm}
    \caption{Running average costs $\hat{\mu}$ and standard deviation $\hat{\sigma}$ for the Monte Carlo simulations from Section \ref{sec:6} for the cost functional \eqref{eq:objective_MonteCarlo}. The upper plot shows the running average costs of nominal DDP over the number of simulations. The lower plot shows the running average costs of robust DDP.}
    \label{fig:RunningAverages}
    \vspace{-4mm}
\end{figure}

To analyze the difference in performance of robust DDP and nominal DDP, we visualize the trajectories from the Monte Carlo experiments in Figure \ref{fig:Trajectories1} and Figure~ \ref{fig:Trajectories2}. Considering particularly the trajectories highlighted in blue, we can see that the nominal policy makes many oscillations before stabilizing the pendulum at the top position. This strategy is sensitive to friction uncertainty, because the friction may dissipate energy from the system faster than it is supplied by the control input. The robust policy, on the other hand, does the swing up faster and leaves then some time for stabilizing the pendulum in the upright position. This strategy is more robust, because there is less time for dissipating energy on this trajectory and in the end, there is time for compensating the disturbances due to friction. What we see is that the changes of the robust strategy make intuitive sense. Moreover, we emphasize that not only the feedback term $\bK_t^2$, but the whole trajectory is changed due to the employment of robust control methods.


\begin{figure}
    \centering
    \includegraphics[width = \linewidth]{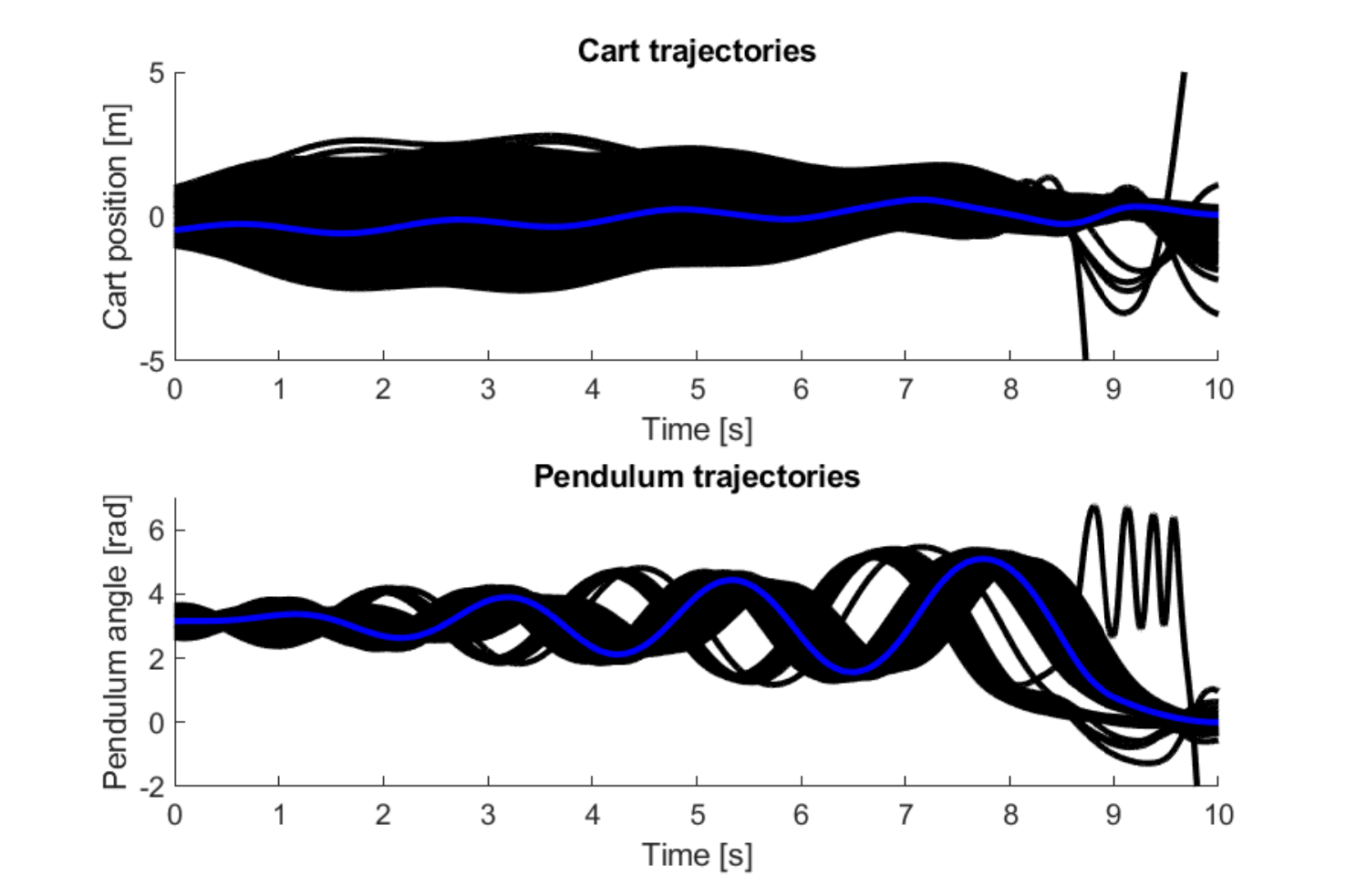}
    \vspace{-8mm}
    \caption{Cart trajectories $s(t)$ and pendulum trajectories $\theta (t)$ for the \textbf{nominal} DDP algorithm from the Monte Carlo simulation of Section \ref{sec:6}. A sample trajectory is highlighted in blue.}
    \label{fig:Trajectories1}
    \vspace{-5mm}
\end{figure}

\begin{figure}
    \centering
    \vspace{-0.5mm}
    \includegraphics[width = \linewidth]{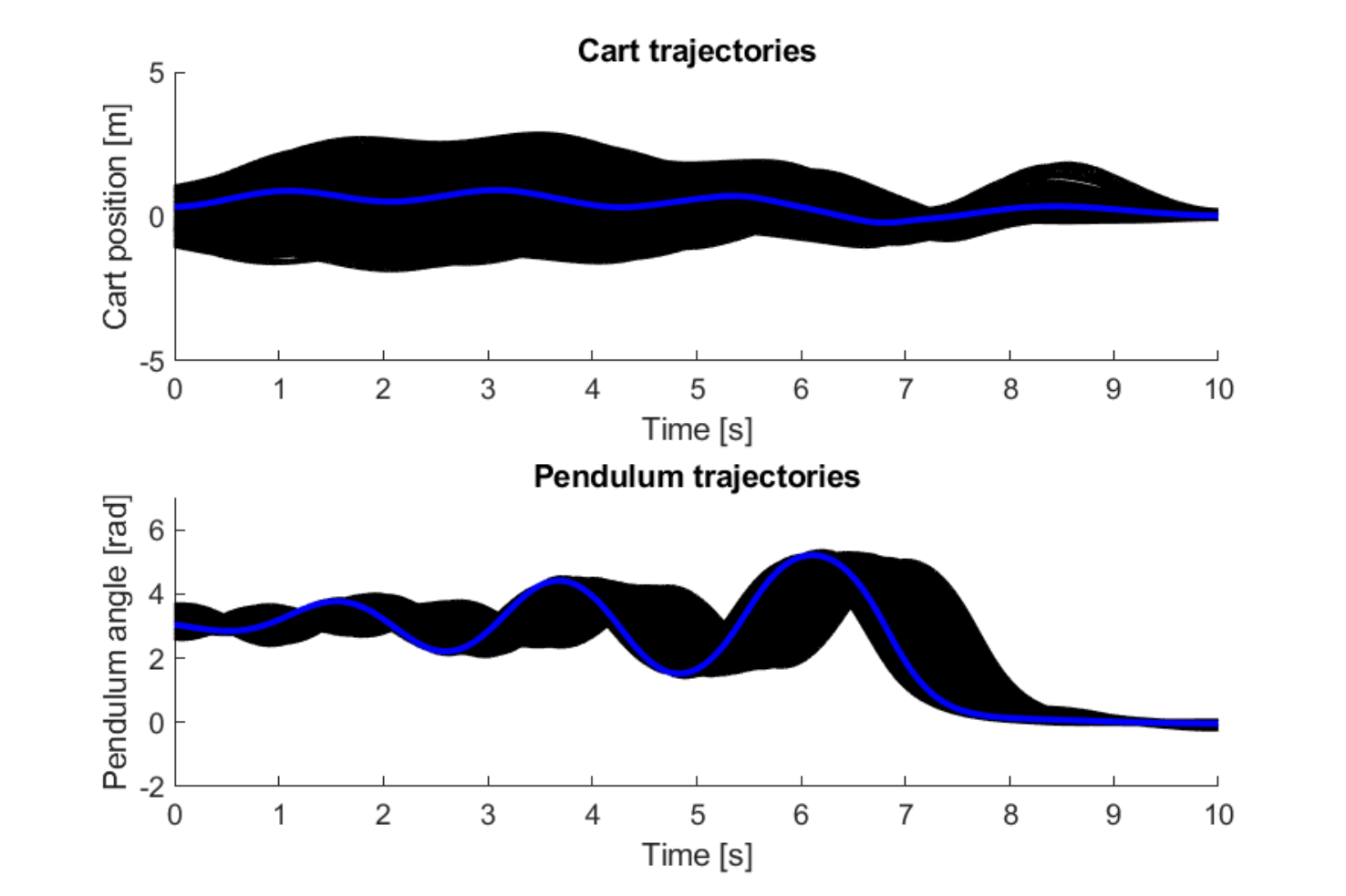}
    \vspace{-8mm}
    \caption{Cart trajectories $s(t)$ and pendulum trajectories $\theta (t)$ for the \textbf{robust} DDP algorithm from the Monte Carlo simulation of Section \ref{sec:6}. A sample trajectory is highlighted in blue.}
    \label{fig:Trajectories2}
    \vspace{-6mm}
\end{figure}

\section{Conclusion}
\label{sec:7}

We combine DDP with robust control and thereby address trajectory generation for LFTs (or nonlinear counterparts thereof). The maximization in the resulting minimax DDP procedure is handled by convex relaxations, such that our method is completely convex for linear systems and a sequential convex programming approach for nonlinear systems. Due to the derivation from DDP, we simultaneously optimize (robustly) over a feedforward term and a linear feedback term. Our numerical example shows that this method has the potential to increase the robust performance of trajectory generation. The ideas presented in the present work should be seen as a first step towards robust planning of feedforward strategies. This could be expanded by investigating IQCs from robust control or step size control and regularization techniques from DDP in the context of robust DDP.

\bibliographystyle{abbrv}
\bibliography{sources.bib}

\end{document}